\documentclass[10pt]{amsart} 
\usepackage{amsmath, amscd, amsthm} 
\usepackage{amssymb, amsfonts}   
\usepackage{enumerate}
\usepackage [english]{babel}
\usepackage [autostyle, english = american]{csquotes}
\MakeOuterQuote{"}
\usepackage[all]{xy} 
\usepackage[margin=1.5in]{geometry}
\usepackage{booktabs}
\usepackage{aliascnt}
\usepackage[svgnames]{xcolor} 
\usepackage{hyperref} 
 \definecolor{dark-red}{rgb}{0.4,0.15,0.15}
\hypersetup{
    colorlinks, linkcolor=dark-red,
    citecolor=DarkBlue, urlcolor=MediumBlue
}

\usepackage[T1]{fontenc}

\newcommand{\Q}{\mathbb{Q}} 

\newcommand{\C}{\mathbb{C}} 
\newcommand{\A}{\mathbb{A}}
\renewcommand{\P}{\mathbb{P}}
\newcommand{\Z}{\mathbb{Z}}
\newcommand{\ZZ}{\mathbb{Z}}
\newcommand{\R}{\mathbb{R}}
\newcommand{\RR}{\R}

\newcommand{\M}{\mathbb{M}}
\newcommand{\F}{\mathbb{F}}
\renewcommand{\AA}{\mathcal{A}}

\newcommand{\iso}{\cong}
\newcommand{\wkeq}{\simeq}

\newcommand{\Sm}{\mathrm{Sm}}

\newcommand{\RRe}{\mathrm{Re}}
\newcommand{\ReC}{\RRe^{C_2}_B}
\newcommand{\comp}[1]{^{\wedge}_{#1}}

\newcommand{\sset}{\mathrm{sSet}}

\newcommand{\SH}{\mathrm{SH}}
\newcommand{\HH}{\mathrm{H}}

\renewcommand{\SS}{\mathbb{S}}

\newcommand{\id}{\mathrm{id}}

\newcommand{\ul}[1]{\underline{\smash{#1}}}

\DeclareMathOperator*{\holim}{\mathrm{holim}}

\DeclareMathOperator{\spec}{\mathrm{Spec}}

\DeclareMathOperator{\Ext}{Ext}

\DeclareMathOperator{\Sym}{Sym}

\numberwithin{equation}{section} 

\theoremstyle{plain}

\newaliascnt{theorem}{equation}  
\newtheorem{theorem}[theorem]{Theorem}  
\aliascntresetthe{theorem}

\newaliascnt{proposition}{equation}  
\newtheorem{proposition}[proposition]{Proposition}
\aliascntresetthe{proposition}

\newaliascnt{lemma}{equation}  
\newtheorem{lemma}[lemma]{Lemma}
\aliascntresetthe{lemma}

\newaliascnt{corollary}{equation}  
\newtheorem{corollary}[corollary]{Corollary}
\aliascntresetthe{corollary}

\newaliascnt{claim}{equation}  

\aliascntresetthe{claim}

\newaliascnt{conjecture}{equation}  

\aliascntresetthe{conjecture}

 \theoremstyle{definition}

\newaliascnt{definition}{equation}  

\aliascntresetthe{definition}

\newaliascnt{example}{equation}  

\aliascntresetthe{example}

\newaliascnt{remark}{equation}  
\newtheorem{remark}[remark]{Remark}
\aliascntresetthe{remark}

\addto\extrasenglish{%

}

\newcommand{\aref}[1]{\autoref{#1}}

\begin{document}
\title[The stable Galois correspondence for real closed fields]{The stable Galois correspondence\\ for real closed fields}
\author{J. Heller}
\address{University of Illinois Urbana-Champaign}
\email{jbheller@illinois.edu}
\author{K. Ormsby}
\address{Reed College}
\email{ormsbyk@reed.edu}
\thanks{The second author was partially supported by National Science Foundation Award DMS-1406327}
\dedicatory{For Ivo and Irving.}

\begin{abstract}
In previous work \cite{HO:E2M}, the authors constructed and studied a lift of the Galois correspondence to stable homotopy categories.  In particular, if $L/k$ is a finite Galois extension of fields with Galois group $G$, there is a functor $c_{L/k}^*:\SH^G\to \SH_k$ from the $G$-equivariant stable homotopy category to the stable motivic homotopy category over $k$ such that 
$c_{L/k}^*(G/H_+) = \spec(L^H)_+$.  The main theorem of \cite{HO:E2M} says that when $k$ is a real closed field and $L=k[i]$, the restriction of $c_{L/k}^*$ to the $\eta$-complete subcategory is full and faithful.  Here we ``uncomplete'' this theorem so that it applies to $c_{L/k}^*$ itself.  Our main tools are Bachmann's theorem on the $(2,\eta)$-periodic stable motivic homotopy category and an isomorphism range for the map $\pi^\R_\star\SS_\R\to \pi^{C_2}_\star\SS_{C_2}$ induced by $C_2$-equivariant Betti realization.
\end{abstract}

\maketitle

\section{Introduction}
In \cite{Levine:comparison}, Levine showed that the ``constant'' functor 
$c^*:\SH\to \SH_k$ from the classical stable homotopy category to the motivic stable homotopy category over an algebraically closed field of characteristic zero is a full and faithful embedding. Inspired by his result, in \cite{HO:E2M} we introduced and studied functors $c^*_{L/k}:\SH^{G}\to \SH_k$, where $L/k$ is a Galois extension with Galois group $G$. We showed that if $k$ is real closed and $L=k[i]$, then 
\emph{after completing} at a prime $p$ and at $\eta$, if $p\neq 2$, the functor $c^*_{L/k}$ is full and faithful. The need for the completion arose from our lack of knowledge about certain homotopy groups of the motivic sphere over $\R$. In the meantime advances have been made. Ananyevskiy-Levine-Panin \cite{alp:witt} established a motivic version of Serre's finiteness theorem, which in particular implies that $c^*_{L/k}$ is full and faithful after $\eta$-completion. 
The purpose of this paper is to use Bachmann's recent results \cite{bachmann:rho}, about a localization of $\pi^\R_\star(\SS_{\R})$, to remove the $\eta$-completions in the main theorem of \cite{HO:E2M}.

\begin{theorem}\label{thm:main}
Let $k$ be a real closed field and $L=k[i]$ be its algebraic closure. Then the functor
$$
c_{L/k}^*:\SH^{C_2}\to \SH_k
$$
is a full and faithful embedding.
\end{theorem}

One of the primary tools in the proof of \aref{thm:main} is the $C_2$-equivariant Betti realization functor $\ReC:\SH_\R\to \SH^{C_2}$ which extends the functor taking a smooth $\R$-scheme to its $\C$-points with complex conjugation action.  In \cite{DI:stems}, Dugger-Isaksen study the effect of this functor on the bigraded homotopy groups of the real motivic sphere spectrum, proving that $\ReC:\pi^\R_{m+n\alpha}(\SS_\R)\comp{2}\to \pi^{C_2}_{m+n\sigma}(\SS_{C_2})\comp{2}$ is an isomorphism when $m\ge 2n-5$.  We use odd-primary Adams spectral sequences and Bachmann's results to produce a range of bigradings in which the integral version of this map is an isomorphism.  In particular, we prove the following.

\begin{theorem}[\aref{thm:betti}]
	The map $\pi^{\R}_{m+n\alpha}(\SS_{\R})\to
	\pi^{C_2}_{m+n\sigma}(\SS_{C_2})$ is
	\begin{enumerate}[(i)]
		\item an injection if $m=2n-6$ and $n\geq 0$, and
		\item  an isomorphism if $m\geq 2n-5$ and $n\geq 0$.
	\end{enumerate}
\end{theorem}

\subsection*{Outline} 
In \aref{sec:prelim} we recall Bachmann's theorem and deduce some consequences for the $C_2$-equivariant Betti realization functor and Morel's $\pm$-splitting of the $2$-periodic stable motivic homotopy category.  In \aref{sec:stems} we adapt the methods of \cite{DI:stems} to odd-primary Adams spectral sequences.  Via an arithmetic fracture square and the results of \aref{sec:prelim}, we deduce \aref{thm:betti}.  Finally, in \aref{sec:proof} we recall how to bootstrap \aref{thm:betti} into a proof of \aref{thm:main}.

\subsection*{Notation}
We use the following notation throughout the paper.
\begin{itemize}
\item $k$ is a field and $L/k$ is a finite Galois extension of fields with Galois group $G$.
\item $\SH_k$ is Voevodsky's stable motivic homotopy category \cite{v:icm}; hom sets in $\SH_k$ are denoted $[~,~]_k$.
\item $\SH^G$ is the $G$-equivariant stable homotopy category in the sense of \cite{lms}; hom sets in $\SH^G$ are denoted $[~,~]_G$.
\item $c^*_{L/k}:\SH^G\to \SH_k$ is the functor induced by the classical Galois correspondence $G/H\mapsto \spec(L^{H})$, constructed in \cite[Section 4.3]{HO:E2M}.\footnote{In \emph{loc.~cit.~}we state that the category $G\sset$ of $G$-simplicial sets is equivalent to the category $\mathrm{sPre}(\mathrm{Or}_G)$ of simplicial presheaves on the orbit category. This isn't quite true: $G\sset\subseteq \mathrm{sPre}(\mathrm{Or}_G)$ is only a full subcategory. (Thanks to Tom Bachmann for drawing our attention to this inaccuracy.) This has no effect on any of the subsequent mathematics in \emph{loc.~cit.~}because what is used is that the associated homotopy categories are equivalent and this is true by Elmendorf's Theorem. }
\item $\SS_k$ is the sphere spectrum in $\SH_k$ and $\SS_G$ is the sphere spectrum in $\SH^G$. 
\item For integers $m,n$, $S^{m+n\alpha} = (S^1)^{\wedge m}\wedge (\A^1\smallsetminus 0)^{\wedge n}$.  If $G=C_2$, $S^{m+n\sigma} = (S^1)^{\wedge m}\wedge (S^\sigma)^{\wedge n}$
where $S^\sigma$ is the one-point compactification of the real sign representation.
\item For $X\in \SH_k$, $\pi^k_{m+n\alpha}X := [S^{m+n\alpha},X]_k$ is the $(m+n\alpha)$-th homotopy group of $X$.
\item For $X\in \SH^{C_2}$, $\pi^{C_2}_{m+n\sigma}X := [S^{m+n\sigma},X]_{C_2}$ is the $(m+n\sigma)$-th homotopy group of $X$.
\item For $X\in \SH_k$, $\pi_*^kX := \bigoplus_{m\in \ZZ}\pi_{m+0\alpha}X$ and similarly for $\pi_*^{C_2}Y$ when $Y\in \SH^G$.
\item Given an embedding $\phi:k\hookrightarrow \RR$, $\RRe_{B,\phi}^{C_2} = \ReC:\SH_k\to \SH^{C_2}$ is the $C_2$-equivariant Betti realization which extends the functor taking a smooth $k$-scheme $X$ to $X(\C)$ with the conjugation action \cite[\S 4.4]{HO:E2M}.
\item Depending on context, $\eta$ is either the motivic Hopf map arising from 
$\A^2\smallsetminus 0\to \P^1$ or the $C_2$-equivariant Hopf map.
\item $(~)\comp{2}$ is the $2$-completion functor and $(~)\comp{\eta}$ is the $\eta$-completion functor.  If $a = 2$ or $\eta$ we have $X\comp{a} = \holim X/a^n$.
\item We write $X[a^{-1}]$ or $X[1/a]$ for the homotopy colimit of $X\xrightarrow{a} X\xrightarrow{a}\cdots$.  If $X\simeq X[a^{-1}]$ we say that $X$ is \emph{$a$-periodic}.
\end{itemize}

\section{Preliminaries}\label{sec:prelim}
The main new input we use in this paper is a recent theorem of Bachmann \cite[Theorem 31]{bachmann:rho}. His theorem compares a localization of $\SH_{\R}$ with the classical stable homotopy category.  The most convenient form of his result for us is the following recasting.
\begin{theorem}[Bachmann]\label{thm:Bman}
 Betti realization induces an equivalence of triangulated categories
 $$
 \RRe^{C_2}_B:\SH_{\R}[1/2,\eta^{-1}]
\xrightarrow{\wkeq} 
\SH^{C_2}[1/2, \eta^{-1}].
 $$
 \end{theorem}
\begin{proof}
 Consider the functor
 $\RRe_{\R}:\SH_{\R}\to \SH$ which extends the  functor $\Sm_{\R}\to {\rm Top}$, sending $X$ to $X(\R)$.  
 Consider as well the composite $\Phi^{C_2}\circ\RRe^{C_2}_{B}$, where 
 $\Phi^{C_2}:\SH^{C_2}\to \SH$ is the geometric fixed points functor.  Both 
 $\RRe_{\R}$ and $\Phi^{C_2}\circ\RRe^{C_2}_{B}$ preserve homotopy colimits and if $X\in \Sm_{\R}$ and $n\in \Z$, they both send $\Sigma^{n}_{\P^1}\Sigma^{\infty}_{\P^1}X_+$ to the spectrum $\Sigma^{n}\Sigma^{\infty}X(\R)_+$. It follows that $\RRe_{\R} = \Phi^{C_2}\circ\RRe^{C_2}_B$ and so we have the commutative triangle of functors
 $$
 \xymatrix{
 \SH_{\R}[1/2,\eta^{-1}] \ar[r]^{\RRe^{C_2}_B}\ar[dr]_{\RRe_\R} & \SH^{C_2}[1/2,\eta^{-1}]\ar[d]^{\Phi^{C_2}}\\
 & \SH[1/2].
 }
 $$
 
Write $\rho:S^{0}\to S^{\sigma}$ for the standard inclusion. Since $\eta^{2}\rho = -2\eta$, we have the equivalence of categories
 $\SH^{C_2}[1/2,\eta^{-1}] \wkeq \SH^{C_2}[1/2,\rho^{-1}]$. It follows that  $\Phi^{C_2}$ is an equivalence
$\SH^{C_2}[1/2,\eta^{-1}] \wkeq \SH[1/2]$. A specialization of Bachmann's theorem \cite[Theorem 31]{bachmann:rho} says that $\RRe_\R$ in the above diagram is an equivalence.  We conclude that $\ReC$ is an equivalence as well.

\end{proof}

Recall Morel's $\pm$-operations in motivic homotopy theory 
(see \cite[Section 16.2]{CD:motives}).  Let $\varepsilon$ denote the stable map induced by the twist isomorphism $S^\alpha\wedge S^\alpha \simeq S^\alpha\wedge S^\alpha$.  In the $C_2$-equivariant setting, let $\varepsilon$ denote the twist $S^\sigma\wedge S^\sigma\simeq S^\sigma\wedge S^\sigma$, and note that $\ReC(\varepsilon)=\varepsilon$.  In either the motivic or equivariant setting, invert $2$ and note that $e_{+} := (\varepsilon - 1)/2$ and $e_- := (\varepsilon + 1)/2$ are orthogonal idempotents.  Let the operation $(~)^+$ denote inversion of $e_+$ and let $(~)^-$ denote inversion of $e_-$, \emph{i.e.}, $(~)^\pm$ is the cofiber of the operation $e_{\pm}$.  For any $2$-periodic motivic or $C_2$-equivariant spectrum $X$ there is a natural splitting $X\simeq X^+\vee X^-$.

\begin{lemma}\label{lemma:split}
If $X$ is a $2$-periodic motivic or $C_2$-equivariant spectrum, then $X^+\simeq X\comp{\eta}$ and $X^-\simeq X[\eta^{-1}]$.  
\end{lemma}
\begin{proof}
We prove the motivic version of this statement, which easily adapts to the $C_2$-equivariant setting.  Let $X$ be a spectrum such that $X\simeq X[1/2]$.  We have $\varepsilon = -1-\rho\eta$ whence $e_- = \frac{-1}{2}\rho\eta$.  Thus inverting $e_-$ inverts both $\rho$ and $\eta$.  Since $(2+\rho\eta)\eta = 0$, this is the same as inverting $2$ and $\eta$, whence $X^- \simeq X[\eta^{-1}]$.  Now apply $\eta$-completion to the splitting $X\simeq X^+\vee X^-$ to get $X\comp{\eta} \simeq (X^+)\comp{\eta}\vee (X^-)\comp{\eta}$.  Since $X^-\simeq X[\eta^{-1}]$, the second summand is trivial.  Since $e_+\eta = 0$, $X^+$ is $\eta$-complete, \emph{i.e.}, $(X^+)\comp{\eta}\simeq X^+$.  We conclude that $X\comp{\eta}\simeq X^+$, as desired.
\end{proof}

As an interesting corollary (which we will not use in the remainder of this paper) we note the following. 

\begin{proposition}\label{lemma:ReEta}
The natural map $\RRe^{C_2}_B((\SS_\RR)\comp{\eta})\to (\SS_{C_2})\comp{\eta}$ is an equivalence.
\end{proposition}
\begin{proof}

Let $\SS=\SS_\RR$ and $\eta$-complete the $2$-primary fracture square for $\SS$ in order to produce the bicartesian square
\[
\xymatrix{
  \SS\comp{\eta}\ar[r]\ar[d] &\SS[1/2]\comp{\eta}\ar[d]\\
  \SS\comp{2,\eta}\ar[r] &(\SS\comp{2}[1/2])\comp{\eta}.
}
\]
Applying $\ReC$ results in a homotopy pullback square which maps to the corresponding fracture square for $(\SS_{C_2})\comp{\eta}$.  The maps between the vertices on the right edge of the square are equivalences by \aref{lemma:split}.  Thus it suffices to show that $\ReC(\SS\comp{2,\eta})\to (\SS_{C_2})\comp{2,\eta}$ is an equivalence.  This may be checked on the level of Mackey functor homotopy groups by comparing the motivic and $C_2$-equivariant Adams spectral sequences as in \cite[Proposition 2.4]{HO:E2M}, concluding the proof.
\end{proof}

\section{Comparing stable stems}\label{sec:stems}
In this section we establish a range of bidegrees in which the map on stable stems $\pi^{\R}_{m+n\alpha}(\SS_{\R}) \to \pi^{C_2}_{m+n\sigma}(\SS_{C_2})$ induced by equivariant Betti realization is  an isomorphism.

Recall that Dugger-Isaksen establish a range in which 
$2$-complete stems are isomorphic. 
\begin{theorem}[{\cite[Theorem 4.1]{DI:stems}}]\label{thm:DI}
	The map $\pi_{m+n\alpha}^{\R}((\SS_{\R})^{\wedge}_{2})\to \pi_{m+n\alpha}^{C_2}((\SS_{C_2})^{\wedge}_{2})$ is an isomorphism if $m\geq 2n-5$ and an injection if $m=2n-6$. 
\end{theorem} 
\begin{remark}
There are isomorphisms $\pi_{m+n\alpha}((\SS_{\R})^{\wedge}_{2})\iso \pi_{m+n\alpha}((\SS_{\R})^{\wedge}_{2,\eta})$ for all $m,n$ 
by \cite[Theorem 1]{HKO}. Similarly there are isomorphisms 
$\pi_{m+n\sigma}^{C_2}((\SS_{C_2})^{\wedge}_{2})\iso \pi_{m+n\sigma}^{C_2}((\SS_{C_2})^{\wedge}_{2,\eta})$ for all $m,n$ by \cite[Theorem 2.10]{HO:E2M}. Dugger-Isaksen's result can thus equivalently be stated as a comparison between $(2,\eta)$-complete stable stems.	
\end{remark}

Recall the discussion of motivic and $C_2$-equivariant cobar complexes from \cite[\S3]{DI:stems}, noting that all of these constructions may be made with $H\F_p$, $p$ odd, in place of $H\F_2$.    The significance of the $p$-primary cobar complex is that it forms the $E_1$-page of the $p$-primary Adams spectral sequence (in the motivic or equivariant context).  In order to concisely express the properties of these spectral sequences, let $\SS$ be $\SS_k$ ($k$ a field) or $\SS_{C_2}$, let $\M_p$ denote the homology of a point with coefficients in $\F_p$ or $\ul\F_p$, let $\AA_p$ denote the $p$-primary motivic or $C_2$-equivariant dual Steenrod algebra, let $C_p^*$ denote the $p$-primary motivic or $C_2$-equivariant cobar complex, let $\gamma$ be $\alpha$ or $\sigma$, depending on context, and let $\Ext_{\AA_p}^{*,*+*\gamma}(\M_p,\M_p)$ denote the homology of $C_p^*$ (which is also $\Ext$ in the category of $\AA_p$-comodules).

\begin{theorem}[{\cite[Theorem 1]{HKO} and \cite[Theorem 2.10]{HO:E2M}}]\label{thm:ass}
The (motivic or $C_2$-equivariant) $p$-primary Adams spectral sequence has $E_1$-page $C_p^*$ and $E_2$-page $\Ext_{\AA_p}^{*,*+*\gamma}(\M_p,\M_p)$; it is strongly convergent with
\[
  E_2^{s,m+n\gamma}\implies \pi_{m-s+n\gamma}\SS\comp{p,\eta}.
\]
When $p=2$, we have $\SS\comp{2,\eta}\simeq \SS\comp{2}$ as long as $\operatorname{cd}_2(k[i])<\infty$ (in the motivic case).
\end{theorem}

The Dugger-Isaksen result is obtained by comparing cobar complexes. We first extend this method to odd $p$ to obtain an isomorphism range on $(p,\eta)$-complete stems. We begin by recalling some facts about the motivic and equivariant Steenrod algebras at odd primes.
Write $\AA_{\F_p}$ for the classical mod-$p$ dual Steenrod algebra. Recall that 
$$
\AA_{\F_p} = \Sym_{\F_p}(\tau_0,\tau_1, \ldots, \xi_1,\xi_2,\ldots)
$$ 
is a free graded commutative algebra, 
where  $|\tau_i| = p^i$ and $|\xi_i|= p^i-1$.

 Recall that $\M_p^\R := \bigoplus_{m,n\in \ZZ} \pi_{m+n\alpha}^\R H\F_p = \F_p[\theta]$ where $\theta$ has degree $2-2\alpha$. This follows from the affirmative resolution of the Bloch-Kato conjecture \cite[Theorem 6.1]{Voevodsky:Z/p} together with \cite[Theorem 7.4]{SV:BK}. In the equivariant case we have  $\M_p^{C_2} := \bigoplus_{m,n\in \ZZ} \pi_{m+n\sigma}^{C_2}H\ul\F_p=\F_p[\theta,\theta^{-1}]$, see \emph{e.g.}~\cite[Theorem 2.8]{Dugger:krss}.
The dual motivic Steenrod algebra over $\RR$ is equal to
$$
\AA_{\R} \iso \M_p^\R\otimes_{\F_p}\AA_{\F_p} = \M_p^{\R}[\tau_0,\tau_1,\ldots, \xi_1, \xi_2,\cdots]/(\tau_0^2,\tau_1^2,\ldots),
$$
where the elements of $\AA_{\F_p}$ are considered to be bigraded by assigning weights so that  $|\tau_i| = p^i + (p^i-1)\alpha$ and $|\xi_i| = (p^i-1) + (p^i-1)\alpha$, see \cite[Remark 12.12]{Voevodsky:reducedpower}.

Similarly, the dual $C_2$-equivariant Steenrod algebra is equal to
$$
\AA_{C_2} \iso \M_p^{C_2}\otimes_{\F_p}\AA_{\F_p} = \M_p^{C_2}[\tau_0,\tau_1,\ldots, \xi_1, \xi_2,\cdots]/(\tau_0^2,\tau_1^2,\ldots),
$$
where in this case elements of $\AA_{\F_p}$ are considered to be bigraded by assigning weights so that  $|\tau_i| = p^i + (p^i-1)\sigma$ and 
$|\xi_i| = (p^i-1) + (p^i-1)\sigma$.

Equivariant Betti realization $\RRe_{B}^{C_2}$ induces maps $\M_p^{\R}\to \M_p^{C_2}$ and $\AA_{\R}\to \AA_{C_2}$ which have the obvious effects on the above named elements, \emph{i.e.}, $\theta\mapsto \theta$, $\tau_i\mapsto\tau_i$  and 
$\xi_i\mapsto\xi_i$.

Write $\M_p$ for either $\M_p^\R$ or $\M_p^{C_2}$. In both cases, the dual Steenrod algebra is free over $\M_p$, and a basis is given by monomials $\tau_0^{\epsilon_0}\tau_1^{\epsilon_1}\cdots \xi_1^{n_1}\xi_2^{n_2}\cdots$, where $\epsilon_i\in \{0,1\}$ and $n_i$ is a nonnegative integer. We write $\tau^{\epsilon}\xi^{n}$ for such a monomial. 
\begin{lemma}\label{lem:coarse}
	Suppose that $|\tau^{\epsilon}\xi^{n}| = k+\ell\alpha$. Then 
	$\displaystyle{k \leq \frac{p}{p-1}\ell +1}$. 
\end{lemma}
\begin{proof}
 The bidegree of $\xi_i$  satisfies 
 $k< \frac{p}{p-1}\ell$ and the bidegree of $\tau_i$ satisfies   
 $k \leq \frac{p}{p-1}\ell$ provided $i\geq 1$. It follows that if $\epsilon_0=0$ then the bidegree of $\tau^{\epsilon}\xi^{n}$ satisfies $k\leq \frac{p}{p-1}\ell$. If $\epsilon_0 =1$ then write $\tau^{\epsilon}\xi^{n} = \tau_0\tau^{\epsilon'}\xi^n$, where $\epsilon_0'=0$. This element thus satisfies the inequality $k \leq \frac{p}{p-1}\ell +1$.  	
\end{proof}
 
\begin{lemma}\label{lem:cobar}
	The map $C^f_\R \to C^f_{C_2}$ is 
	\begin{enumerate}[(i)]
		\item an injection in all degrees, and 
		\item an isomorphism if 
		$k-f\geq \left\lfloor\frac{p}{p-1}\ell  -\frac{4p-2}{p-1}\right\rfloor +1 $. 
	\end{enumerate}
\end{lemma}
\begin{proof}
	We have that $C^f_{C_2} \iso \M_p^{C_2}\otimes_{M_p^{\R}}C^f_{\R}$, that
		 $C^f_{\R}$ is free over $\M_p^\R$, and that the map $\M^p_{\R}\to \M^{p}_{C_2}$ is injective. It follows that $C^f_\R \to C^f_{C_2}$ is injective.

		 Let $Z = [z_1 | z_2 | \cdots | z_f]$ be a cobar element of Adams filtration $f$, where each $z_i$ is of the form 
		 $\tau^{\epsilon}\xi^{n}$, of bidegree $k_i+\ell_i\alpha$. By \aref{lem:coarse} we have that $k_i\leq \frac{p}{p-1}\ell_i+1$. Summing over $i$, we find that if $|Z| = k+\ell\alpha$, then 
		 $k\leq \frac{p}{p-1}\ell +f$. The cokernel of $C^f_\R \to C^f_{C_2}$ consists of elements of the form
		 $\theta^{-n}Z$ for $n\geq 1$. The elements $\theta^{-n}$ lie above the line of slope\footnote{We use the standard convention in which the $k$-axis is horizontal and the $\ell\alpha$-axis is vertical.} $(p-1)/p$ passing through $\theta^{-1}$ and so we find that these satisfy the inequality 
		 $k\leq \frac{p}{p-1}\ell - \frac{4p-2}{p-1}$. It follows that the bidegree of 
		 $\theta^{-n}Z$ satisfies the inequality 
		 $k\leq \frac{p}{p-1}\ell + f - \frac{4p-2}{p-1}$. Thus the cokernel is zero in bidegrees satisfying $k> \frac{p}{p-1}\ell + f - \frac{4p-2}{p-1}$.  

\end{proof}

Recall the form of the motivic and $C_2$-equivariant Adams spectral sequences from \aref{thm:ass}.  
Equivariant Betti realization induces a map between these spectral sequences which we can now analyze. 

\begin{proposition}\label{prop:ext}
	The map 
	$\Ext^{(f, k+\ell\alpha)}_{\AA_\R}(\M_p^{\R}, \M_p^{\R}) \to 
	\Ext^{(f, k+\ell\alpha)}_{\AA_{C_2}}(\M_p^{C_2}, \M_p^{C_2}) $
	is an injection if $k-f= \left\lfloor\frac{p}{p-1}\ell  -\frac{4p-2}{p-1}\right\rfloor$ and an isomorphism if 
	$k-f\geq \left\lfloor\frac{p}{p-1}\ell  -\frac{4p-2}{p-1}\right\rfloor +1$.
\end{proposition}
\begin{proof}
	This follows from \cite[Lemma 3.4]{DI:stems} and \aref{lem:cobar}.
\end{proof}

\begin{lemma}
	$\Ext^{(f, k+\ell\alpha)}_{\AA_\R}(\M_p^{\R}, \M_p^{\R})$ is a finite-dimensional $\F_p$-vector space for all $f,k,\ell$. 
\end{lemma}
\begin{proof}
	Writing $C_{\F_p}$ for the classical cobar complex, we have $C_{\R} \iso \M^{\R}_p\otimes_{\F_p} C_{\F_p}$. The universal coefficient theorem then yields the isomorphism 
	$$
	\Ext_{\AA_{\R}}^{*,*+*\alpha}(\M_p^\R,\M_p^\R)\iso \Ext_{\AA_{\F_p}}^{*,*}(\F_p,\F_p)\otimes_{\F_p}\M^\R_p
	$$
	up to a grading shift, and this is a finite-dimensional $\F_p$-vector space in all degrees.
\end{proof}

\begin{theorem}\label{thm:isorange}
	The map $\pi^{\R}_{m+n\alpha}((\SS_{\R})^{\wedge}_{(p,\eta)})\to
	\pi^{C_2}_{m+n\sigma}((\SS_{C_2})^{\wedge}_{(p,\eta)})$ is
	\begin{enumerate}[(i)]
	\item an injection if 
	$m=\left\lfloor\frac{p}{p-1}n  -\frac{4p-2}{p-1}\right\rfloor$, and 	
	\item an isomorphism if 
	$m\geq \left\lfloor\frac{p}{p-1}n  -\frac{4p-2}{p-1}\right\rfloor +1$.
	\end{enumerate}
\end{theorem}
\begin{proof}
This follows from \aref{prop:ext} using the same argument as in \cite[Theorem 4.1]{DI:stems}.
\end{proof}

\begin{remark}
In \cite{DI:stems},	Dugger-Isaksen establish a second isomorphism range on $2$-complete  stems. Namely, $\pi_{m+n\alpha}(\SS_{\R})^{\wedge}_{2} \to \pi_{m+n\sigma}(\SS_{C_2})^{\wedge}_{2}$ is an isomorphism if $m+n\leq -1$. A version of this range appears to hold on $(p,\eta)$-complete stems, with the difference that the map might be only surjective when $m+n=-1$. However, this second isomorphism range doesn't extend to integral stems and we do not pursue it further here.
\end{remark}

We now turn our attention to $p$-complete spheres.
\begin{proposition}\label{prop:oddp}
The map $\pi^{\R}_{m+n\alpha}((\SS_{\R})^{\wedge}_{p})\to
\pi^{C_2}_{m+n\sigma}((\SS_{C_2})^{\wedge}_{p})$ is
\begin{enumerate}
\item an injection if $m=\left\lfloor\frac{p}{p-1}n  -\frac{4p-2}{p-1}\right\rfloor$, and 	
\item an isomorphism if $m\geq \left\lfloor\frac{p}{p-1}n  -\frac{4p-2}{p-1}\right\rfloor +1$.
\end{enumerate}
\end{proposition}
\begin{proof}
Write $\SS$ for either of $\SS_{\R}$ or $\SS_{C_2}$. If $p=2$, the statement of the proposition is Dugger-Isaksen's \aref{thm:DI}, so we can assume $p$ is odd. In this case $2$ is invertible in $\SS/p^r$ and in $\SS^\wedge_p$. By 
\aref{lemma:split} we have that 
$(\SS^{\wedge}_{p})^+ \wkeq \SS^{\wedge}_{(p,\eta)}$ and  $(\SS^{\wedge}_{p})^- \wkeq \SS^{\wedge}_{p}[\eta^{-1}]$ 
and so we have that 
$$
\SS^\wedge_p\wkeq \SS^{\wedge}_{(p,\eta)}\vee \SS^{\wedge}_{p}[\eta^{-1}].
$$ 
Note that we have an isomorphism 
$(\SS^\wedge_p)^{-} \iso \lim_r ((\SS/p^r)^{-})$. It follows from 
\aref{thm:Bman} that the map $\pi_{*+*\alpha}^{\R}((\SS_{\R})^{\wedge}_{p}[\eta^{-1}])\to \pi_{*+*\sigma}^{C_2}((\SS_{C_2})^{\wedge}_{p}[\eta^{-1}])$ is an isomorphism. 
The result thus follows from \aref{thm:isorange} 
and the direct sum decomposition of $\SS^{\wedge}_p$ above.

\end{proof}

\begin{theorem}\label{thm:betti}
	The map $\pi^{\R}_{m+n\alpha}(\SS_{\R})\to
	\pi^{C_2}_{m+n\sigma}(\SS_{C_2})$ is
	\begin{enumerate}[(i)]
		\item an injection if $m=2n-6$ and $n\geq 0$, and
		\item  an isomorphism if $m\geq 2n-5$ and $n\geq 0$.
	\end{enumerate}
\end{theorem}
\begin{proof}
	Consider the comparison of long exact sequences of homotopy groups induced by cofiber sequences  
	$$
	\SS \to  	
	\Big(\prod_{p} \SS^{\wedge}_{p}\Big)\vee  \SS_{\Q} \to \Big(\prod_p \SS^{\wedge}_{p}\Big)_{\Q}.
	$$
obtained from the arithmetic fracture squares for $\SS_{\R}$ and $\SS_{C_2}$. 
It suffices to show that the comparison map $\pi_{m+n\alpha}^{\R}(-)\to \pi_{m+n\sigma}^{C_2}(-)$ at the middle and the righthand terms are injections if $m=2n-6$ and $n\geq 0$ and
an isomorphism if $m\geq 2n-5$ and $n\geq 0$.

The inequalities for the isomorphism range of $p$-complete stable stems from \aref{prop:oddp} for odd $p$ are dominated by Dugger-Isaksen's inequalities in  \aref{thm:DI}, for the $2$-complete stable stems. It follows that 
$\pi^\R_{m+n\alpha}(\prod_{p} (\SS_{\R})^{\wedge}_{p})\to   \pi^{C_2}_{m+n\alpha}(\prod_{p} (\SS_{C_2})^{\wedge}_{p})$ is an injection if  $m=2n-6$ and an isomorphism if $m\geq 2n-5$. Since we have that the map $\pi^\R_{m+n\alpha}(\prod_{p} (\SS_{\R})^{\wedge}_{p})_{\Q}\to   \pi^{C_2}_{m+n\alpha}(\prod_{p} (\SS_{C_2})^{\wedge}_{p})_{\Q}$ is a filtered colimit of these maps, it too is an injection if $m=2n-6$ and an isomorphism if $m\geq 2n-5$.

By \aref{thm:Bman} and \aref{lemma:split}, $\pi^\R_{*+*\alpha}(\SS_{\R})_{\Q}^-\to \pi^{C_2}_{*+*\alpha}(\SS_{C_2})_{\Q}^-$ is an isomorphism. 
By \cite[Theorems 11 and 16.2.13]{CD:motives}, $(\SS_{\R})_{\Q}^{+}\wkeq \HH\Q$, where $\HH\Q$ is the rationalized motivic cohomology spectrum. Thus, we have that
$\pi_{i +j\alpha}^{\R}(\SS_{\R})^+_{\Q} = 0$ whenever $j>0$. When $j=0$, we have that 
$\pi_{i}^{\R}(\SS_{\R})^+_{\Q} = 0$ if $i\neq 0$ and $\pi_{0}^{\R}(\SS_{\R})^+_\Q=\Q$. 
We have 
$$
\pi_{i +j\alpha}^{C_2}(\SS_{C_2})^+_{\Q} = \begin{cases}
\Q & j\textrm{ is even and } i+j = 0 \\
0 & \textrm{else}.
\end{cases}
$$
Note that if $j\geq 0$, this vanishing region satisfies $i\geq 2j-5$. The map $\pi_{0}^{\R}(\SS_{\R})^+_\Q\to \pi_{0}^{C_2}(\SS_{C_2})^+_\Q$ is an ismorphism. We conclude that $\pi_{m+n\alpha}^{\R}(\SS_{\R})^+_\Q\to \pi_{m+n\sigma}^{C_2}(\SS_{C_2})^+_\Q$ is  an injection if $m=2n-6$ and $n\geq 0$, and
  an isomorphism if $m\geq 2n-5$ and $n\geq 0$.
\end{proof}

\section{Proof of {\aref{thm:main}}}\label{sec:proof}
We finish by explaining how the comparison of stable stems in the previous section implies the embedding theorem. 

\begin{proposition}\label{prop:b}
	If 
	\begin{enumerate}
		\item[(i)] $\RRe^{C_{2}}_{B}:  [S^n,\SS_{\R}]_{\R} \xrightarrow{\iso} 
		[S^n,  \SS_{C_{2}}]_{C_{2}}$, and
		\item[(ii)] $\RRe^{C_{2}}_{B}:  [\spec(\C)_+\wedge S^n,\SS_{\R}]_{\R} \xrightarrow{\iso} 
		[C_{2\,+}\wedge S^n,  \SS_{C_{2}}]_{C_{2}}$
	\end{enumerate}
	are isomorphisms for all $n\in\Z$, then \aref{thm:main} is true for any real closed field $k$. 
\end{proposition}
\begin{proof}
	Let $k$ be a real closed field and $L=k[i]$. To prove \aref{thm:main}, it suffices to prove that 
	\begin{enumerate}
		\item[(a)]  $c^*_{L/k}:[S^n,\SS_{C_2}]_{C_2}
		\xrightarrow{\iso} [S^n, \SS_{k}]_k$, and 
		\item[(b)] $c^*_{L/k}:
		[C_{2\,+}\wedge S^n, \SS_{C_2}]_{C_2}\xrightarrow{\iso} 
		[\spec(L)_+\wedge S^n, \SS_{k}]_k$
	\end{enumerate}
	are isomorphisms for all $n\in \Z$, by the same argument as in the beginning of the proof of \cite[Theorem 2.21]{HO:E2M}.	
	
	To prove that the maps in (a), (b) are isomorphisms, we can assume that $k=\R$ and $L = \C$, by the same argument as in \cite[Proposition 2.20]{HO:E2M}. We now consider the  $C_2$-equivariant Betti realization functor
	$\RRe^{C_2}_{B}:\SH_{\R} \to \SH^{C_2}$. Since $\RRe_{B}^{C_2}\circ c^*_{\C/\R} = \id$, it follows that (a) and (b) are isomorphisms.  
\end{proof}

\begin{corollary}[\aref{thm:main}]
	Let $k$ be a real closed field and $L=k[i]$ be its algebraic closure. Then the functor
	$$
	c_{L/k}^*:\SH^{C_2}\to \SH_k
	$$
	is a full and faithful embedding.
\end{corollary} 
\begin{proof}
	If $i<0$ then $\pi^\R_i(\SS_\R) = \pi_i^{C_2}(\SS_{C_2})=0$ and so the map in \ref{prop:b}(i) is an isomorphism for $i<0$. It is an isomorphism for $i\geq 0$ by setting $n=0$ in \aref{thm:betti}. 
	The map in \ref{prop:b}(ii) is identical to the map 
	$[ S^n,\SS_{\C}]_{\C} \to 
	[S^n,  \SS]$ induced by complex Betti realization. This is an isomorphism by Levine's theorem \cite{Levine:comparison}. 
\end{proof}

\bibliographystyle{plain}
\bibliography{E2M}

\end{document}